\documentclass[12pt]{article}
\usepackage{amssymb}
\usepackage{graphics}

\usepackage{url}

\usepackage{hyperref}

\usepackage[authoryear,round]{natbib}
\begin{document}

\newenvironment {proof}{{\noindent\bf Proof.}}{\hfill $\Box$ \medskip}

\newtheorem{theorem}{Theorem}[section]
\newtheorem{lemma}[theorem]{Lemma}
\newtheorem{condition}[theorem]{Condition}
\newtheorem{proposition}[theorem]{Proposition}
\newtheorem{remark}[theorem]{Remark}
\newtheorem{hypothesis}[theorem]{Hypothesis}
\newtheorem{corollary}[theorem]{Corollary}
\newtheorem{example}[theorem]{Example}
\newtheorem{definition}[theorem]{Definition}

\renewcommand {\theequation}{\arabic{section}.\arabic{equation}}
\def \non{{\nonumber}}
\def \hat{\widehat}
\def \tilde{\widetilde}
\def \bar{\overline}

\title{\large{{\bf Weak and strong solutions of general stochastic models }}}
                                                       
\author{{\small \begin{tabular}{l}                               
Thomas G. Kurtz \thanks{Research supported in part by NSF grant
DMS 11-06424}  \\   
Departments of Mathematics and Statistics  \\       
University of Wisconsin - Madison  \\                   
480 Lincoln Drive \\                                                          
Madison, WI  53706-1388    \\                         
kurtz@math.wisc.edu      \\
\url{http://www.math.wisc.edu/~kurtz/}  \\\                       
\end{tabular}} }

\date{December 29, 2013}

\maketitle

\begin{abstract}

Typically, a stochastic model relates stochastic ``inputs'' 
and, perhaps, controls to stochastic ``outputs''.  A general 
version of the Yamada-Watanabe and Engelbert theorems 
relating existence and uniqueness of weak and strong 
solutions of stochastic equations is given in this context.  
A notion of {\em compatibility\/} between inputs and outputs is 
critical in relating the general result to its classical 
forebears.  The relationship between the compatibility 
condition and the usual formulation of stochastic 
differential equations driven by semimartingales is 
discussed.

\vspace{.1in}

\noindent {\bf Key words:}  weak solution, strong solution, 
stochastic models, pointwise uniqueness, pathwise 
uniqueness, compatible solutions, stochastic differential 
equations, stochastic partial differential equations, 
backward stochastic differential equations, Meyer-Zheng 
conditions, Jakubowski topology

\vspace{.1in}

\noindent {\bf MSC 2010 Subject Classification:}  Primary:  60G05
Secondary:  60H10, 60H15, 60H20, 60H25.

\end{abstract}

\setcounter{equation}{0}

\section{Introduction and main theorem}
This paper is essentially a rewrite of \cite{Kur07} 
following a realization that the general, abstract theorem 
in that paper was neither as abstract as it could be nor 
as general as it should be.  The reader familiar with 
the earlier paper may not be pleased by the greater
abstraction, but an example indicating the value of the 
greater generality will be given in Section \ref{comp}.  
To simplify matters for the reader,  proofs of 
several lemmas that originally appeared in the earlier 
paper are included, but the reader should refer to the 
earlier paper for more examples and additional 
references.

As with the results of the earlier paper, the main 
theorem given here generalizes the famous theorem of 
\cite{YW71} giving the relationship between weak and 
strong solutions of an It\^o equation for a diffusion and 
their existence and uniqueness.  A second reason for this 
rewrite is that the main observation ensuring that
the main theorem gives the Yamada-Watanabe result is buried in 
a proof in the earlier paper.  Here it is stated separately 
as Lemma \ref{compcpl}.

The motivation of the original
Yamada-Watanabe result arises naturally in the process 
of proving existence of solutions of a stochastic 
differential equation or, in the context of the present 
paper, existence of a stochastic model determined by 
constraints that may but need not be equations.  The 
basic existence argument starts by identifying a 
sequence of approximations to the equation (or model) 
for which existence of solutions is simple to prove, 
proving relative compactness of the sequence of 
approximating solutions, and then verifying that any 
limit point is a solution of the original equation (model).  
The issue addressed by the Yamada-Watanabe theorem is 
that frequently, the kind of compactness verified is 
weak or distributional compactness.  Consequently, what 
can be claimed about the limit is that there exists a 
probability space on which processes are defined that 
satisfy the original equation.  Such solutions are called 
{\em weak solutions}, and their existence leaves open the 
question of whether there exists a solution on every 
probability space that supports the stochastic inputs of 
the model, that is, the Brownian motion and initial 
position in the original It\^o equation context.  The 
assertion of the Yamada-Watanabe theorem and Theorem 
\ref{gyw} below is that if a strong enough form of 
uniqueness can be verified, then existence of a weak 
solution implies existence on every such probability 
space.  

A stochastic model describes the relationship between 
stochastic inputs and stochastic outputs.  For example, 
in the case of the It\^o equation, 
\[X(t)=X(0)+\int_0^t\sigma (X(s))dW(s)+\int_0^tb(X(s))ds,\]
$X(0)$ and $W$ are the stochastic inputs and the solution $X$ 
gives the outputs.  Typically, the distribution of the 
inputs is specified (for example, the initial distribution 
is given and $X(0)$ is assumed independent of the 
Brownian motion $W$), and the model is determined by a 
set of constraints (possibly, but not necessarily, 
equations) that relate the inputs to the outputs.  In the 
general setting here, the inputs will be given by a 
random variable $Y$ with values in a complete, separable 
metric space $S_2$ and the outputs $X$ will take values in a 
complete, separable metric space $S_1$.  For the It\^o 
equation, we could take $S_2={\Bbb R}^d\times C_{{\Bbb R}^d}[0,\infty 
)$ and 
$S_1=C_{{\Bbb R}^d}[0,\infty )$.  

Let ${\cal P}(S_1\times S_2)$ be the space of probability measures on 
$S_1\times S_2$, and for random variables $(X,Y)$ in $S_1\times S_
2$, let 
$\mu_{X,Y}\in {\cal P}(S_1\times S_2)$ denote their joint distribution.  Our 
model is determined by specifying a distribution $\nu$ for 
the inputs $Y$ and a set of constraints $\Gamma$ relating $X$ and $
Y$.  
Let ${\cal P}_{\nu}(S_1\times S_2)$ be the set of $\mu\in {\cal P}
(S_1\times S_2)$ such that 
$\mu (S_1\times\cdot )=\nu$, and let ${\cal S}_{\Gamma ,\nu}$ be the subset of $
{\cal P}_{\nu}(S_1\times S_2)$ 
such that $\mu_{X,Y}\in {\cal S}_{\Gamma ,\nu}$ implies $(X,Y)$ meets the constraints 
in $\Gamma$.  Of course, since we are not placing any 
restriction on the nature of the constraints, ${\cal S}_{\Gamma ,
\nu}$ could 
be any subset of ${\cal P}_{\nu}(S_1\times S_2)$.  

For a second example, consider a typical stochastic 
optimization problem.

\begin{example}{\rm
Suppose $\Gamma_0$ is a collection of constraints of the form
\[E[\psi (X,Y)]<\infty\mbox{\rm \  and  }E[f_i(X,Y)]=0,\quad i\in 
{\cal I},\]
where $\psi\geq 0$ and $|f_i(x,y)|\leq\psi$.

Let $0\leq c(x,y)\leq\psi (x,y)$, and let $\Gamma$ be the set of 
constraints
obtained from $\Gamma_0$ by adding the requirement
\[\int c(x,y)\mu (dx\times dy)=\inf_{\mu'\in {\cal S}_{\Gamma_0,\nu}}
\int c(x,y)\mu'(dx\times dy).\]
It is natural to ask if the infimum is achieved with $X$ of 
the form $X=F(Y)$. \hfill $\Box$
}
\end{example}

In the terminology of \cite{Eng91} and \cite{Jac80}, 
$\mu\in {\cal S}_{\Gamma ,\nu}$ is a {\em joint solution measure\/} for our model $
(\Gamma ,\nu )$.
A {\em weak solution\/} (or simply a {\em solution\/}) for $(\Gamma 
,\nu )$ is any pair of 
random variables $(X,Y)$ defined on any probability space 
such that $Y$ has distribution $\nu$ and $(X,Y)$ meets the 
constraints in $\Gamma$, that is, $\mu_{X,Y}\in {\cal S}_{\Gamma 
,\nu}$.
We have the following definition for a strong solution.

\begin{definition}\ \label{strdef}
A solution $(X,Y)$ for $(\Gamma ,\nu )$ is a {\em strong solution\/} if
there exists a Borel measurable function $F:S_2\rightarrow S_1$ such 
that $X=F(Y)$ a.s.
\end{definition}

If a strong solution exists on some probability space, 
then a strong solution exists for any $Y$ with distribution 
$\nu$.  It is important to note that being a strong solution 
is a distributional property, that is, the joint 
distribution of $(X,Y)$ is determined by $\nu$ and $F$.
The following lemma helps to clarify the difference 
between a strong solution and a weak solution that does 
not correspond to a strong solution.

\begin{lemma}\label{sol}
Let $\mu\in {\cal P}_{\nu}(S_1\times S_2)$.
\begin{itemize}
\item[a)]There exists a 
transition function $\eta$ such that 
$\mu (dx\times dy)=\eta (y,dx)\nu (dy)$.

\item[b)]There exists a Borel measurable $G:S_2\times [0,1]\rightarrow 
S_1$ such that if $Y$ has 
distribution $\nu$ and $\xi$ is independent of $Y$ and uniformly 
distributed on $[0,1]$, $(G(Y,\xi ),Y)$ has distribution $\mu$.

\item[c)]$\mu$ corresponds to a strong solution if 
and only if $\eta (y,dx)=\delta_{F(y)}(dx)$.
\end{itemize}
\end{lemma}

\begin{proof}
Statement (a) is a standard result on the 
disintegration of measures.  A particularly nice 
construction that gives the desired $G$ in Statement (b) can 
be found in \cite{BD83}.  Statement (c) is immediate.
\end{proof}

We have the following notions of uniqueness.

\begin{definition}
{\em Pointwise (pathwise for stochastic processes) uniqueness }
holds, if $X_1$, $X_2$, and $Y$ defined on the same probability 
space with $\mu_{X_1,Y},\mu_{X_2,Y}\in {\cal S}_{\Gamma ,\nu}$ implies $
X_1=X_2$ a.s.

{\em Joint uniqueness in law\/} (or {\em weak joint uniqueness\/}) holds, 
 if ${\cal S}_{\Gamma ,\nu}$ contains at most one measure.  

{\em Uniqueness in law\/} (or {\em weak uniqueness\/}) holds if all 
$\mu\in {\cal S}_{\Gamma ,\nu}$ have the same marginal distribution on $
S_1$.
\end{definition}

We have the following generalization of the theorems of 
\cite{YW71} and \cite{Eng91}.

\begin{theorem}\label{gyw}
The following are equivalent:
\begin{itemize}

\item[a)]${\cal S}_{\Gamma ,\nu}\neq\emptyset$, and pointwise uniqueness holds. 

\item[b)]There exists a strong solution, and joint 
uniqueness in law holds.
\end{itemize}
\end{theorem}

\begin{remark}
In the special case that all constraints are given by 
simple equations, for example,
\begin{equation}f_i(X,Y)=0\quad a.s.\quad i\in {\cal I},\label{simpeq}\end{equation}
then Proposition 2.10 of \cite{Kur07} shows that 
pointwise uniqueness, joint uniqueness in law, and 
uniqueness in law are equivalent.  Note that stochastic 
differential equations are not of the form (\ref{simpeq}) 
(see Section \ref{comp}) since (\ref{simpeq}) does not 
involve any adaptedness requirements.  Consequently, 
the equivalence of 
uniqueness in law and joint uniqueness in law does not 
follow from this proposition in that setting; however, 
\cite{Che03} has shown the equivalence of uniqueness in 
law and joint uniqueness in law for It\^o 
equations for diffusion processes.
\end{remark}

\begin{proof}
Assume (a).  If $\mu_1,\mu_2\in {\cal S}_{\Gamma ,\nu}$, then there exist Borel 
measurable functions $G_1(y,u)$ and 
$G_2(y,u)$ on $S_2\times [0,1]$ such that for $Y$ with distribution $
\nu$ and $\xi_1,\xi_2$ 
uniform on $[0,1]$, all independent, $(G_1(Y,\xi_1),Y)$ has 
distribution $\mu_1$ and $(G_2(Y,\xi_2),Y)$ has distribution $\mu_
2$.
By pointwise uniqueness,
\[G_1(Y,\xi_1)=G_2(Y,\xi_2)\quad a.s.\]
From the independence of $\xi_1$ and $\xi_2$, it follows that there 
exists a Borel measurable $F$ on $S_2$ such that 
$F(Y)=G_1(Y,\xi_1)=G_2(Y,\xi_2)$ a.s.  (See Lemma A.2 of 
\cite{Kur07}.)  

Assume (b). Suppose 
  $X_1$, $X_2$, $Y$ are 
defined on the same probability space and 
$\mu_{X_1,Y},\mu_{X_2,Y}\in {\cal S}_{\Gamma ,\nu}$.   
By Lemma \ref{sol}, the unique $\mu\in {\cal S}_{\Gamma ,\nu}$ must satisfy 
$\mu (dx\times dy)=\delta_{F(y)}(dx)\nu (dy)$, so $X_1=F(Y)=X_2$ almost 
surely giving pointwise uniqueness.
\end{proof}

The main result in \cite{Kur07}, Theorem 3.14, was 
stated assuming the compatibility condition to be discussed in the 
next section and under the assumption that ${\cal S}_{\Gamma ,\nu}$ was 
convex.  Neither assumption is needed for Theorem 
\ref{gyw}.  The compatibility 
condition is critical to showing that Theorem \ref{gyw}
implies the classical Yamada-Watanabe result as well as 
a variety of more recent results for other kinds of 
stochastic equations.  (See \cite{Kur07} for references.)  
The convexity assumption is useful in giving the 
following additional result.

\begin{corollary}\label{evstrong} 
Suppose ${\cal S}_{\Gamma ,\nu}$ is nonempty and convex.  Then every 
solution is a strong solution if and only if pointwise 
uniqueness holds.  
\end{corollary}

\begin{proof}
By Theorem \ref{gyw}, pointwise uniqueness  
implies ${\cal S}_{\Gamma ,\nu}$ contains only one distribution and the 
corresponding solution is strong.  Conversely, suppose 
every solution is a strong solution.
If $\mu_1,\mu_2\in {\cal S}_{\Gamma ,\nu}$, then $\mu_0=\frac 12\mu_
1+\frac 12\mu_2\in {\cal S}_{\Gamma ,\nu}$.  Let $Y$  have
distribution $\nu$.  Then there exist Borel Functions $F_1$ and 
$F_2$ such that $(F_1(Y),Y)$ has distribution $\mu_1$ and $(F_2(Y
),Y)$ 
has distribution $\mu_2$.  Let 
 $\xi$ be uniformly distributed on $[0,1]$ and 
independent of $Y$. Define
\[X=\left\{\begin{array}{cl}
F_1(Y)&\quad\xi >1/2\\
F_2(Y)&\quad\xi\leq 1/2.\end{array}
\right.\]
Then $(X,Y)$ has distribution $\mu_0$ and must satisfy $X=F(Y)$ 
a.s. 
for some $F$.  Since $\xi$ is independent 
of $Y$, we must have $F_1(Y)=F(Y)=F_2(Y)$ $a.s.,$ giving 
pointwise uniqueness.
\end{proof}

\setcounter{equation}{0} 
\section{Compatibility}\label{comp} It is not immediately 
obvious that Theorem \ref{gyw} gives the classical 
Yamada-Watanabe theorem since proofs of pathwise 
uniqueness require appropriate adaptedness conditions in 
order to compare two solutions.  This leads us to 
introduce the notion of compatibility.  In what follows, 
if $S$ is a metric space, then ${\cal B}(S)$ will denote the Borel 
$\sigma$-algebra and $B(S)$ will denote the space of bounded, Borel 
measurable functions; if ${\cal M}$ is a $\sigma$-algebra, $B({\cal M}
)$ will 
denote the space of bounded, ${\cal M}$-measurable functions.

Let $E_1$ and $E_2$ be complete, separable metric spaces, and 
let $D_{E_i}[0,\infty )$, be the Skorohod space of cadlag $E_i$-valued 
functions.  Let $Y$ be a process in $D_{E_2}[0,\infty )$.  By ${\cal F}_
t^Y$, we 
mean the completion of $\sigma (Y(s),s\leq t)$.  

\begin{definition}\label{cptdef1}
A process $X$ in $D_{E_1}[0,\infty )$ is 
{\em temporally compatible\/} with $Y$ if for each $t\geq 0$ and $h\in B(D_{
E_2}[0,\infty ))$,
\begin{equation}E[h(Y)|{\cal F}^{X,Y}_t]=E[h(Y)|{\cal F}^Y_t]\label{compid}\end{equation}
where $\{{\cal F}^{X,Y}_t\}$ denotes the complete filtration generated by  $
(X,Y)$ 
and $\{{\cal F}_t^Y\}$ denotes the complete filtration generated by $
Y$.
\end{definition}

This definition is essentially (4.5) of \cite{Jac80} which 
is basic to the statement of Theorem 8.3 of that paper 
which gives a version of the Yamada-Watanabe theorem 
for general stochastic differential equations driven by 
semimartingales.
If $Y$ has independent increments, then $X$ is compatible 
with $Y$ if $Y(t+~\cdot )-Y(t)$ is independent of ${\cal F}_t^{X,
Y}$ for all 
$t\geq 0$. (See Lemma \ref{indcnd}\ below.)

We will consider a more general notion of compatibility. 
 If ${\cal B}_{\alpha}^{S_1}$ is a 
sub-$\sigma$-algebra of  ${\cal B}(S_1)$ and $X$ is an $S_1$-valued random 
variable on a complete probability space $(\Omega ,{\cal F},P)$, 
then ${\cal F}_{\alpha}^X\equiv\mbox{\rm the completion of }\{\{X
\in D\}:D\in {\cal B}_{\alpha}^{S_1}\}$ is the 
complete, sub-$\sigma$-algebra of ${\cal F}$ 
generated by $\{h(X):h\in B({\cal B}_{\alpha}^{S_1})\}$.  ${\cal F}_{
\alpha}^Y$ is defined similarly 
for a sub-$\sigma$-algebra ${\cal B}_{\alpha}^{S_2}\subset {\cal B}
(S_2)$.

\begin{definition}\label{cptdef2}
Let ${\cal A}$ be an index set, and for each $\alpha\in {\cal A}$, let $
{\cal B}_{\alpha}^{S_1}$ be a 
sub-$\sigma$-algebra of ${\cal B}(S_1)$ and ${\cal B}_{\alpha}^{S_
2}$ be a sub-$\sigma$-algebra of 
${\cal B}(S_2)$.  
The collection ${\cal C}\equiv \{({\cal B}_{\alpha}^{S_1},{\cal B}_{
\alpha}^{S_2}):\alpha\in {\cal A}\}$ will be referred to as a 
{\em compatibility structure}.

Let $Y$ be an $S_2$-valued random variable.
An $S_1$-valued random variable $X$ is ${\cal C}$-{\em compatible\/} with $
Y$ if 
for each $\alpha\in {\cal A}$ and each $h\in B(S_2)$ (or equivalently, each 
$h\in L^1(\nu )$),
\begin{equation}E[h(Y)|{\cal F}^X_{\alpha}\vee {\cal F}_{\alpha}^
Y]=E[h(Y)|{\cal F}^Y_{\alpha}]\label{compid2}\end{equation}
\end{definition}

\begin{remark}
Temporal compatibility, as defined above,
 is a special case of compatibility, and we will reserve 
this terminology for the case in which $\{{\cal F}_t^X\}$ and $\{
{\cal F}_t^Y\}$ are 
the complete filtrations generated by $X$ and $Y$.  Of 
course, in this setting ${\cal F}_t^{X,Y}={\cal F}_t^X\vee {\cal F}_
t^Y$.  

Compatibility 
conditions do arise that have index set ${\cal A}=[0,\infty )$ but which 
are not temporal compatibility.  For example, for a
time-change equation
\[X(t)=Y(\int_0^t\beta (X(s))ds),\]
the natural compatibility condition sets
\[{\cal F}_{\alpha}^Y=\mbox{\rm \ the completion of }\sigma (Y(u)
:0\leq u\leq\alpha )\]
but takes
\[{\cal F}_{\alpha}^X=\mbox{\rm \ the completion of }\sigma (\{\int_
0^t\beta (X(s))ds\leq r\}:r\leq\alpha ,t\geq 0).\]
In this case, compatibility ensures that $\tau (t)=\int_0^t\beta 
(X(s))ds$ 
is a stopping time with respect to the filtration 
$\{{\cal F}_{\alpha}^X\vee {\cal F}_{\alpha}^Y,\alpha\geq 0\}$.
\end{remark}

\begin{lemma}\label{indcnd}
Suppose that for each $\alpha\in {\cal A}$ there exist random variables 
$(Y_{\alpha},Y^{\alpha})$ with values in some measurable space $R_{
\alpha}\times R^{\alpha}$ such that 
$\sigma (Y)=\sigma (Y_{\alpha},Y^{\alpha})$, $Y_{\alpha}$ is ${\cal F}_{
\alpha}^Y$-measurable, and $Y^{\alpha}$ is 
independent of ${\cal F}^X_{\alpha}\vee {\cal F}_{\alpha}^Y$.  Then $
X$ is compatible with $Y$
\end{lemma}

\begin{proof}
If $h\in B(S_2)$, then there exist $h_{\alpha}\in B(R_{\alpha}\times 
R^{\alpha})$ such that 
$h(Y)=h_{\alpha}(Y_{\alpha},Y^{\alpha})$ a.s.  Then
\begin{eqnarray*}
E[h(Y)|{\cal F}_{\alpha}^X\vee {\cal F}_{\alpha}^Y]&=&E[h_{\alpha}
(Y_{\alpha},Y^{\alpha})|{\cal F}_{\alpha}^X\vee {\cal F}_{\alpha}^
Y]\\
&=&E[\int_{R^{\alpha}}h_{\alpha}(Y_{\alpha},y)\mu_{Y^{\alpha}}(dy
)|{\cal F}_{\alpha}^X\vee {\cal F}_{\alpha}^Y]\\
&=&\int_{R^{\alpha}}h_{\alpha}(Y_{\alpha},y)\mu_{Y^{\alpha}}(dy)\\
&=&E[\int_{R^{\alpha}}h_{\alpha}(Y_{\alpha},y)\mu_{Y^{\alpha}}(dy
)|{\cal F}_{\alpha}^Y]\end{eqnarray*}
\end{proof}

In the temporal setting, \citet*{BER05} employ a 
a condition that requires every $\{{\cal F}_t^Y\}$-martingale to be a 
$\{{\cal F}_t^{X,Y}\}$-martingale.  More generally,  $\{{\cal F}_{
\alpha}^Y,\alpha\in {\cal A}\}$ is a {\em filtration\/} if
 ${\cal A}$ is partially ordered and $\alpha_1\prec\alpha_2$ implies $
{\cal F}^Y_{\alpha_1}\subset {\cal F}_{\alpha_2}^Y$.
We consider the following condition.

\begin{condition}\label{mgcnd}\ $\{{\cal F}_{\alpha}^Y,\alpha\in 
{\cal A}\}$ and $\{{\cal F}_{\alpha}^X,\alpha\in {\cal A}\}$ are 
filtrations and every $\{{\cal F}_{\alpha}^Y\}$-martingale is a $
\{{\cal F}_{\alpha}^Y\vee {\cal F}_{\alpha}^X\}$ 
martingale.
\end{condition}

\begin{lemma}\label{mglem} 
If $\{{\cal F}_{\alpha}^Y,\alpha\in {\cal A}\}$ and $\{{\cal F}_{
\alpha}^X,\alpha\in {\cal A}\}$ are filtrations, then 
${\cal C}$-compatibility implies Condition \ref{mgcnd}.  
\end{lemma}

\begin{remark}
The earlier paper (\cite{Kur07}) and the 
original version of the current paper casually claimed 
equivalence of the martingale condition and compatibility.  
A referee has pointed out that the claim 
was not only casual, but false.  Condition \ref{mgcnd}\ 
gives an example of what we will call {\em partial }
{\em compatibility conditions}, that is, (\ref{compid2}) holds for 
a subset of $h\in L^1(\nu )$.  Partial compatibility conditions 
will be discussed further in Section \ref{partcmp}.
\end{remark}

\begin{proof}
Let $\{M(\alpha ),\alpha\in {\cal A}\}$ be a $\{{\cal F}_{\alpha}^
Y\}$-martingale.  For each $\alpha\in {\cal A}$, 
there exists a Borel function $h_{\alpha}$ such that 
$M(\alpha )=h_{\alpha}(Y)$ a.s.  Suppose $\alpha_1\prec\alpha_2$.  Then
\[E[M(\alpha_2)|{\cal F}^X_{\alpha_1}\vee {\cal F}_{\alpha_1}^Y]=
E[h_{\alpha_2}(Y)|{\cal F}^X_{\alpha_1}\vee {\cal F}_{\alpha_1}^Y
]=E[h_{\alpha_2}(Y)|{\cal F}_{\alpha_1}^Y]=M(\alpha_1).\]
\end{proof}

Note that (\ref{compid2}) is equivalent to requiring that 
for each $h\in B(S_2)$,
\begin{equation}\inf_{f\in B({\cal B}_{\alpha}^{S_1}\times {\cal B}_{
\alpha}^{S_2})}E[(h(Y)-f(X,Y))^2]=\inf_{f\in B({\cal B}_{\alpha}^{
S_2})}E[(h(Y)-f(Y))^2],\label{cmpid2}\end{equation}
so compatibility is a property of the joint 
distribution of $(X,Y)$. Consequently, compatibility is a 
constraint on joint distributions.  To emphasize the 
special role of compatibility, ${\cal S}_{\Gamma ,{\cal C},\nu}$ will denote the 
collection of joint distributions that satisfy the 
constraints in $\Gamma$ and the ${\cal C}$-compatibility constraint.

\begin{example}\label{smgsde}{\rm
Let $U$ be a process in $D_{{\Bbb R}^d}[0,\infty )$, $V$ an ${\Bbb R}^
m$-valued 
semimartingale with respect to the filtration $\{{\cal F}_t^{U,V}
\}$, and 
$H:D_{{\Bbb R}^d}[0,\infty )\rightarrow D_{{\Bbb M}^{d\times m}}[
0,\infty )$ (${\Bbb M}^{d\times m}$ the space of 
$d\times m$-dimensional matrices) be Borel measurable and 
satisfy $H(x,t)=H(x(\cdot\wedge t),t)$ for all $x\in D_{{\Bbb R}^
d}[0,\infty )$ and 
$t\geq 0$.  Then $X$ is defined to be a solution of 
\begin{equation}X(t)=U(t)+\int_0^tH(X,s-)dV(s)\label{smgsde1}\end{equation}
if $X$ is temporally compatible with $Y=(U,V)$ (ensuring 
that the stochastic integral exists) and 
\[\lim_{n\rightarrow\infty}E[1\wedge |X(t)-U(t)-\sum_kH(X,\frac k
n)(V(\frac {k+1}n\wedge t)-V(\frac kn\wedge t))|]=0,\quad t\geq 0
.\]
Note that this definition assumes more regularity than 
is necessary or is assumed in \cite{Jac80}.
}
\end{example}

To prove pointwise (pathwise) uniqueness, we still need 
 some way of comparing compatible solutions. 

\begin{definition}\label{jntcmp}
Let the random variables $X_1$, $X_2$, and $Y$ be defined on the 
same probability space with $X_1$ and $X_2$ $S_1$-valued and $Y$ 
$S_2$-valued.  $(X_1,X_2)$ are {\em jointly} ${\cal C}$-{\em compatible\/} with $
Y$ if 
\[E[h(Y)|{\cal F}^{X_1}_{\alpha}\vee {\cal F}^{X_2}_{\alpha}\vee 
{\cal F}^Y_{\alpha}]=E[h(Y)|{\cal F}_{\alpha}^Y],\quad\alpha\in {\cal A}
,h\in B(S_2).\]
(Note that if $(X_1,X_2)$ are jointly ${\cal C}$-compatible with $
Y$, 
then each of $X_1$ and $X_2$ is ${\cal C}$-compatible with $Y$.)

{\em Pointwise uniqueness for jointly} ${\cal C}$-{\em compatible solutions\/} holds 
if for every triple of processes $(X_1,X_2,Y)$ 
defined on the same probability space such that 
$\mu_{X_1,Y},\mu_{X_2,Y}\in {\cal S}_{\Gamma ,{\cal C},\nu}$ and $
(X_1,X_2)$ is jointly 
compatible with $Y$, $X_1=X_2$ a.s.
\end{definition}

With reference to Lemma \ref{indcnd}, uniqueness for 
jointly temporally compatible solutions is the usual kind 
of uniqueness considered for stochastic differential 
equations driven by Brownian motion, L\'evy processes, 
and/or Poisson random measures.  For example, let
$Y=(X(0),Z)$, where $Z$ is a L\'evy process. Consider the 
equation
\[X(t)=X(0)+\int_0^tH(X(s-))dZ(s),\]
where we require $X$ and $Z$  to be adapted to a filtration 
$\{{\cal F}_t\}$ such that $Z(t+\cdot )-Z(t)$ is independent of $
{\cal F}_t$, $t\geq 0$.  If 
there exist two such solutions with $X_1(0)=X_2(0)=X(0)$ 
adapted to $\{{\cal F}_t\}$, then since 
${\cal F}_t^{X_1}\vee {\cal F}_t^{X_2}\vee {\cal F}_t^Z\subset {\cal F}_
t$,  
\begin{eqnarray*}
&&E[h(Z(t+\cdot )-Z(t),Z(\cdot\wedge t))|{\cal F}_t^{X_1}\vee {\cal F}_
t^{X_2}\vee {\cal F}_t^Z]\\
&&\qquad = E[E[h(Z(t+\cdot )-Z(t),Z(\cdot\wedge 
t))|{\cal F}_t]|{\cal F}_t^{X_1}\vee {\cal F}_t^{X_2}\vee {\cal F}_
t^Z]\\
&&\qquad =E[\int h(z,Z(\cdot\wedge t))\mu_{Z(t+\cdot )-Z(t)}(dz)|{\cal F}_
t^{X_1}\vee {\cal F}_t^{X_2}\vee {\cal F}_t^Z]\\
&&\qquad = \int h(z,Z(\cdot\wedge t))\mu_{Z(t+\cdot )-Z(t)}(dz)\\
&&\qquad = E[\int h(z,Z(\cdot\wedge t))\mu_{Z(t+\cdot )-Z(t)}(dz)|{\cal F}_
t^Z\vee\sigma (X(0))],\end{eqnarray*}
which gives the joint compatibility of $X_1$ and $X_2$ with 
$(X(0),Z)$.

The following lemma ensures that pointwise uniqueness 
of jointly compatible solutions 
is equivalent to the notion of 
pointwise uniqueness used in Theorem \ref{gyw}\ and 
hence, for example, Theorem \ref{gyw}\ implies the 
classical Yamada-Watanabe theorem.

\begin{lemma}\label{uequiv}
Pointwise uniqueness for jointly ${\cal C}$-compatible solutions in 
${\cal S}_{\Gamma ,{\cal C},\nu}$ is equivalent to pointwise uniqueness in $
{\cal S}_{\Gamma ,{\cal C},\nu}$.
\end{lemma}

Recall that for
 $\mu_1,\mu_2\in {\cal S}_{\Gamma ,{\cal C},\nu}$ and  $Y$, $\xi_
1$, and $\xi_2$  independent, $Y$ 
with distribution $\nu$  
and $\xi_1$ and $\xi_2$ uniform on $[0,1]$, there exist Borel 
measurable
$G_1:S_2\times [0,1]\rightarrow S_1$ and $G_2:S_2\times [0,1]\rightarrow 
S_1$ such that $(G_1(Y,\xi_1),Y)$ 
has distribution $\mu_1$ and $(G_2(Y,\xi_2),Y)$ has distribution $
\mu_2$.  

Clearly pointwise uniqueness in ${\cal S}_{\Gamma ,{\cal C},\nu}$ implies pointwise 
uniqueness for jointly ${\cal C}$-compatible solutions.  The 
converse follows by repeating the reasoning in the proof 
of Theorem \ref{gyw}\ now using the following lemma.  

\begin{lemma}\label{compcpl}
If $\mu_1,\mu_2\in {\cal S}_{\Gamma ,{\cal C},\nu}$ and $(G_1(Y,\xi_
1),Y)$ has distribution $\mu_1$ and 
$(G_2(Y,\xi_2),Y)$ has distribution $\mu_2$, where $\xi_1$ and $\xi_
2$ are 
independent and independent of $Y$, then $G_1(Y,\xi_1),G_2(Y,\xi_
2)$ 
are jointly compatible with $Y$. 
\end{lemma}$ $

In order to prove Lemma \ref{compcpl}, we need the 
following technical lemma.

\begin{lemma}\label{adcomp}
 $X$ is ${\cal C}$-compatible with $Y$ if and only if for each $\alpha
\in {\cal A}$ and 
each $g\in B({\cal B}_{\alpha}^{S_1})$,
\begin{equation}E[g(X)|Y]=E[g(X)|{\cal F}_{\alpha}^Y]\label{adid}\end{equation}
\end{lemma}

\begin{proof}
Suppose that $X$ is ${\cal C}$-compatible with $Y$.  Then for $f\in 
B(S_2)$ 
and $g\in B({\cal B}_{\alpha}^{S_1})$,
\begin{eqnarray*}
E[f(Y)g(X)]&=&E[E[f(Y)|{\cal F}^X_{\alpha}\vee {\cal F}_{\alpha}^
Y]g(X)]\\
&=&E[E[f(Y)|{\cal F}^Y_{\alpha}]g(X)]\\
&=&E[E[f(Y)|{\cal F}^Y_{\alpha}]E[g(X)|{\cal F}^Y_{\alpha}]]\\
&=&E[f(Y)E[g(X)|{\cal F}_{\alpha}^Y]],\end{eqnarray*}
and (\ref{adid}) follows.
Conversely, for $f\in B(S_2)$, 
$g\in B({\cal B}_{\alpha}^{S_1})$, and $h\in B({\cal B}_{\alpha}^{
S_2})$, we have 
\begin{eqnarray*}
E[E[f(Y)|{\cal F}_{\alpha}^Y]g(X)h(Y)]&=&E[E[f(Y)|{\cal F}_{\alpha}^
Y]E[g(X)|{\cal F}_{\alpha}^Y]h(Y)]\\
&=&E[f(Y)E[g(X)|Y]h(Y)]\\
&=&E[f(Y)g(X)h(Y)],\end{eqnarray*}
and compatibility follows.
\end{proof}

\begin{proof}{\bf [of Lemma \ref{compcpl}]}
For $f\in B({\cal B}_{\alpha}^{S_1})$, by the independence of $\xi_
2$ from $(Y,\xi_1)$ and
Lemma \ref{adcomp},
\[E[f(G_1(Y,\xi_1))|Y,\xi_2]=E[f(G_1(Y,\xi_1))|Y]=E[f(G_1(Y,\xi_1
))|{\cal F}_{\alpha}^Y].\]
Consequently, for $X_1=G_1(Y,\xi_1)$, $X_2=G_2(Y,\xi_2)$, $f\in B
(S_2)$, 
$g_1,g_2\in B({\cal B}_{\alpha}^{S_1})$, and $h\in B({\cal B}_{\alpha}^{
S_2})$, 
\begin{eqnarray*}
&&E[f(Y)g_1(X_1)g_2(X_2)h(Y)]\\
&&\qquad =E[f(Y)E[g_1(X_1)|Y,\xi_2]g_2(X_2)h(Y)]\\
&&\qquad =E[f(Y)E[g_1(X_1)|{\cal F}^Y_{\alpha}]g_2(X_2)h(Y)]\\
&&\qquad =E[E[f(Y)|{\cal F}^{X_2}_{\alpha}\vee {\cal F}_{\alpha}^
Y]E[g_1(X_1)|{\cal F}^Y_{\alpha}]g_2(X_2)h(Y)]\\
&&\qquad =E[E[f(Y)|{\cal F}^Y_{\alpha}]E[g_1(X_1)|Y,\xi_2]g_2(X_2
)h(Y)]\\
&&\qquad =E[E[f(Y)|{\cal F}^Y_{\alpha}]g_1(X_1)g_2(X_2)h(Y)],\end{eqnarray*}
giving the joint compatibility.
\end{proof}

Lemma \ref{adcomp} also gives the following result.

\begin{proposition}\label{adpt}
If $X$ is a strong, compatible solution, then ${\cal F}^X_{\alpha}
\subset {\cal F}_{\alpha}^Y$ for 
each $\alpha\in {\cal A}$.  (In particular, in the temporal compatibility 
setting, $X$ is adapted to the filtration $\{{\cal F}_t^Y\}$.)  Conversely, 
if ${\cal F}_{\alpha}^X\subset {\cal F}_{\alpha}^Y$ for each $\alpha
\in {\cal A}$ and $\sigma (X)\subset\vee_{\alpha\in {\cal A}}{\cal F}_{
\alpha}^X$, then $X$ is 
a strong, compatible solution.
\end{proposition}

\begin{proof}
Since $X=F(Y)$, by (\ref{adid}), for each $g\in B({\cal B}_{\alpha}^{
S_1})$,
\[g(X)=g(F(Y))=E[g(F(Y))|Y]=E[g(X)|Y]=E[g(X)|{\cal F}_{\alpha}^Y]
\quad a.s.\]
Consequently, $g(X)$ is ${\cal F}_{\alpha}^Y$-measurable and hence 
${\cal F}_{\alpha}^X\subset {\cal F}_{\alpha}^Y$.

Conversely, the assumption that ${\cal F}_{\alpha}^X\subset {\cal F}_{
\alpha}^Y$ for each $\alpha\in {\cal A}$ 
implies $X$ is compatible with $Y$, and the additional 
assumption implies
\[\sigma (X)\subset\vee_{\alpha\in {\cal A}}{\cal F}_{\alpha}^X\subset
\vee_{\alpha\in {\cal A}}{\cal F}_{\alpha}^Y\subset\sigma (Y),\]
so there exists a Borel measurable function $F$ such that 
$X=F(Y)$ a.s.
\end{proof}

\begin{example}{\rm
McKean-Vlasov limits lead naturally to stochastic 
differential equations of the form
\begin{equation}X(t)=X(0)+\int_0^t\sigma (X(s),\mu_{X(s)})dW(s)+\int_
0^tb(X(s),\mu_{X(s)})ds\label{mv}\end{equation}
where $\mu_{X(s)}$ is required to be the distribution of $X(s)$.  
Alexander Veretennikov raised the question of a 
Yamada-Watanabe type result for equations of this form.  
Setting $Y=(X(0),W)$ and requiring temporal compatibility, 
the set of joint solution measures ${\cal S}_{\Gamma ,{\cal C},\nu}$ may not be 
convex.  Consequently, the results of \cite{Kur07} may
not apply.  Theorem \ref{gyw}, however, does not 
assume convexity of ${\cal S}_{\Gamma ,{\cal C},\nu}$, and hence weak existence 
and pathwise uniqueness imply the existence of a strong 
solution of (\ref{mv}).  } 
\end{example}

\section{Partial compatibility and existence of compatible 
solutions.}\label{partcmp}  Let ${\cal H}\subset B(S_2)$ (or ${\cal H}
\subset L^1(\nu )$).
We will say that a random variable $X$ is 
$({\cal C},{\cal H})$-{\em partially compatible\/} with $Y$ if (\ref{compid2}) holds for 
each $h\in {\cal H}$ but not necessarily for all
$h\in B(S_2)$. We could handle partial compatibility conditions 
the same way we handled compatibility conditions if the 
analog of Lemma  \ref{compcpl} held.  Unfortunately, that 
is not in general the case.

\begin{example}{\rm
Let $\zeta_1,\ldots ,\zeta_4$ be independent with distribution 
$P\{\zeta_i=1\}=P\{\zeta_i=-1\}=\frac 12$, and let 
\[Y=(Y_1,Y_2,Y_3,Y_4)=(\zeta_1\zeta_2,\zeta_2\zeta_3,\zeta_3\zeta_
4,\zeta_4\zeta_1).\]
Note that any three of the components are independent 
but the four are not.  Assume that the index set ${\cal A}$
consists of a single element $\alpha$. Let ${\cal F}_{\alpha}^Y=\sigma 
(Y_1)$, and for $\xi$ 
independent of $Y$ and uniformly distributed on $[0,1]$, let
\[X=G(Y,\xi )={\bf 1}_{\{\xi <\frac 12\}}Y_2+{\bf 1}_{\{\xi\geq\frac 
12\}}Y_3,\]
and ${\cal F}_{\alpha}^X=\sigma (G(Y,\xi ))$.  For $h_0(Y)=Y_4$, 
\[E[h_0(Y)|{\cal F}^Y_{\alpha}\vee {\cal F}_{\alpha}^X]=0=E[h_0(Y
)|{\cal F}_{\alpha}^Y],\]
so $X$ is $({\cal C},{\cal H})$-partially compatible with $Y$ for $
{\cal H}=\{h_0\}$.  
However, if $\xi_1$ and $\xi_2$ are independent, uniform $[0,1]$ 
random variables and we define
\[X_1=G(Y,\xi_1)\mbox{\rm \  and  }X_2=G(Y,\xi_2),\]
then
\[E]h_0(Y)|{\cal F}^Y_{\alpha}\vee {\cal F}^{X_1}_{\alpha}\vee {\cal F}_{
\alpha}^{X_2}]={\bf 1}_{\{X_1\neq X_2\}}Y_1X_1X_2+\frac 13{\bf 1}_{
\{X_1=X_2\}}Y_1,\]
and the corresponding joint partial compatibility 
condition fails. }
\end{example}

Since every compatible solution will satisfy any partial 
compatibility condition, pointwise (pathwise) uniqueness 
proved under a partial compatibility condition will give 
pointwise uniqueness under the compatibility condition.  
This observation is relevant not only under Condition 
\ref{mgcnd}\ but also for the general stochastic 
differential equation given in Example \ref{smgsde}.

Uniqueness results for equations of the form 
(\ref{smgsde1}) are usually proved under the 
assumption that solutions $X_1$ and $X_2$ and $Y=(U,V)$ are 
adapted to a filtration $\{{\cal F}_t\}$ under which $V$ is a 
semimartingale.  $V$ can always be written as $V=M+A$, 
where $M$ is a local martingale with jumps bounded by $1$ 
and $A$ is a finite variation process.  The localizing 
sequence for $M$ can be taken to be 
$\tau_n=\inf\{t:\sup_{s\leq t}|M(s)|\geq n\}$, and an appropriate joint 
partial compatibility condition follows from the 
observation that for $t>s$,
\begin{eqnarray*}
E[M(t\wedge\tau_n)|{\cal F}_s^{X_1}\vee {\cal F}_s^{X_2}\vee {\cal F}_
s^Y]&=&E[E[M(t\wedge\tau_n)|{\cal F}_s]|{\cal F}_s^{X_1}\vee {\cal F}_
s^{X_2}\vee {\cal F}_s^Y]\\
&=&M(s\wedge\tau_n)\\
&=&E[M(t\wedge\tau_n)|{\cal F}_s^Y].\end{eqnarray*}
Consequently, pathwise uniqueness results in settings of 
this form imply pathwise uniqueness for jointly 
compatible solutions.  

To apply Theorem \ref{gyw} when pointwise uniqueness 
is known under partial compatibility conditions still 
requires existence of a compatible solution since we do 
not 
have the analog of Lemma \ref{compcpl} for partial 
compatibility.  The following lemma gives a general 
approach to the required existence.

\begin{lemma}\label{cmpexist}
Suppose there exist $C_{\alpha}^{S_2}\subset C_b(S_2)$ and $C_{\alpha}^{
S_1}\subset C_b(S_1)$ such 
that ${\cal B}_{\alpha}^{S_2}=\sigma (g\in C_{\alpha}^{S_2})$ and $
{\cal B}_{\alpha}^{S_1}=\sigma (g\in C_{\alpha}^{S_1})$.  (Without loss 
of generality, we can assume $C_{\alpha}^{S_1}$ and $C_{\alpha}^{
S_2}$ are algebras.) 
Suppose 
$(X_n,Y)\in S_1\times S_2$, $X_n$ is ${\cal C}$-compatible with $
Y$, $(X_n,Y)\Rightarrow (X,Y)$.  Then 
$X$ is ${\cal C}$-compatible with $Y$. 
\end{lemma}

\begin{remark}\label{ascnt}
With reference to the continuous mapping theorem (for 
example, \cite{EK86}, Corollary 3.1.9), the continuity 
assumption on the functions generating ${\cal B}_{\alpha}^{S_1}$ and $
{\cal B}_{\alpha}^{S_2}$ can 
be weakened.  For ${\cal B}_{\alpha}^{S_1}$, it is enough for the functions $
g$ 
to be continuous almost everywhere with respect to $\mu_X$, 
and for ${\cal B}_{\alpha}^{S_2}$, the functions $g$ only need to be continuous 
almost everywhere with respect to $\mu_Y$.  This observation 
is particularly relevant for cadlag processes since the 
evaluation function $x\in D_E[0,\infty )\rightarrow x(t)\in E$ is not 
continuous, but it will be almost everywhere continuous 
for the process of interest provided $t$ is not a fixed 
point of discontinuity, that is, provided 
$P\{X(t)\neq X(t-)\}=0$.

In many settings, natural approximations for 
a solution will satisfy $X_n=F_n(Y)$ and ${\cal F}^{X_n}_{\alpha}
\subset {\cal F}_{\alpha}^Y$ 
and hence will be strong, 
compatible solutions of approximating models. (See 
Proposition \ref{adpt}.)

\end{remark}

\begin{proof}
For $f\in C_b(S_1)$, $g_1\in C_{\alpha}^{S_1}$ and $g_2\in C_{\alpha}^{
S_2}$
\[E[f(Y)g_1(X_n)g_2(Y)]=E[E[f(Y)|{\cal F}_{\alpha}^Y]g_1(X_n)g_2(
Y)].\]
Since $C_b(S_2)$ is dense in $L^1(\nu )$, for each $\alpha$ and $
\epsilon >0$, there exists 
$f_{\alpha ,\epsilon}\in C_b(S_2)$ such that 
\[E[|E[f(Y)|{\cal F}_{\alpha}^Y]-f_{\alpha ,\epsilon}(Y)|]\leq\epsilon 
.\]
Consequently, it follows that
\begin{eqnarray*}
\lim_{n\rightarrow\infty}E[f(Y)g_1(X_n)g_2(Y)]&=&E[f(Y)g_1(X)g_2(
Y)]\\
&=&\lim_{n\rightarrow\infty}E[E[f(Y)|{\cal F}_{\alpha}^Y]g_1(X_n)
g_2(Y)]\\
&=&\lim_{\epsilon\rightarrow 0}\lim_{n\rightarrow\infty}E[f_{\alpha 
,\epsilon}(Y)g_1(X_n)g_2(Y)]\\
&=&\lim_{\epsilon\rightarrow 0}E[f_{\alpha ,\epsilon}(Y)g_1(X)g_2
(Y)]\\
&=&E[E[f(Y)|{\cal F}_{\alpha}^Y]g_1(X)g_2(Y)]\end{eqnarray*}
verifying compatibility.
\end{proof}

Note that in the proof of the above lemma, we use the 
fact that $Y$, or more precisely, the distribution of $Y$, 
does not depend on $n$ in order to obtain the $f_{\alpha ,\epsilon}$.

Problems do arise in which input processes have fixed 
points of discontinuity and the application of Lemma 
\ref{cmpexist}\ is problematic even with the observation 
made in Remark \ref{ascnt}.  The following definition of 
RC-compatibility (or more precisely, RC-temporal 
compatibility) avoids this problem.  It looks strange, but 
Lemma \ref{rcprd}\ shows that it is equivalent to a more 
natural assumption. $M_E[0,\infty )$ denotes the collection of Borel 
measurable functions $x:[0,\infty )\rightarrow E$. $S_i$ could be $
D_{E_i}[0,\infty )$ 
under the usual Skorohod topology, but other spaces can 
be useful.  (See Example \ref{bw}.)

\begin{definition}
Let ${\cal A}=\{(t,\epsilon ):t\in [0,\infty ),\epsilon >0\}$, $S_
1\subset M_{E_1}[0,\infty )$, and 
$S_2\subset M_{E_2}[0,\infty )$.  For $\alpha =(t,\epsilon )$, define
\[C_{\alpha}^{S_2}=\{\int_s^{s+r}g(x(u))du:s\leq t,0<r<\epsilon ,
g\in C_b(E_2)\]
and 
\[C_{\alpha}^{S_1}=\{\int_{(s-r)\vee 0}^sg(x(u))ds:s\leq t,0<r<\epsilon 
,g\in C_b(E_1)\},\]
and set ${\cal B}_{\alpha}^{S_2}=\sigma (g\in C_{\alpha}^{S_2})$ and $
{\cal B}_{\alpha}^{S_1}=\sigma (g\in C_{\alpha}^{S_1})$.  Then 
${\cal C}_{RC}\equiv \{({\cal B}_{\alpha}^{S_1},{\cal B}_{\alpha}^{
S_2}):\alpha\in {\cal A}\}$ defines the {\em RC-compatibility }
{\em structure\/} (RC for ``right continuous'') on 
$(M_{E_1}[0,\infty ),M_{E_2}[0,\infty ))$.
\end{definition}

Note that if $S_1=D_{E_1}[0,\infty )$ and $S_2=D_{E_2}[0,\infty )$, then $
C_{\alpha}^{S_1}$ 
and $C_{\alpha}^{S_2}$ are collections of continuous functions and 
Lemma \ref{cmpexist}\ applies to RC-compatibility.

Assume that $X$ and $Y$ are right continuous, and let $\{{\cal F}_
t^X\}$ 
and $\{{\cal F}_t^Y\}$ denote their natural filtrations.  Note that for 
$t>0$, ${\cal F}_{(t,\epsilon )}^X=$${\cal F}_{t-}^X\equiv\vee_{s
<t}{\cal F}_s^X$,  $\cap_{\epsilon >0}{\cal F}^Y_{(t,\epsilon )}=
{\cal F}_{t+}^Y\equiv\cap_{s>t}{\cal F}_s^Y$, and 
${\cal F}^Y_{(t,\epsilon )}={\cal F}^Y_{t+\epsilon -}$.  We have the following lemma.  

\begin{lemma}\label{rcprd}
Let $X$ be a right continuous, $E_1$-valued process 
and $Y$ be a right continuous, $E_2$-valued process.  
Then $X$ is RC-compatible with $Y$ if and only if
\begin{equation}E[h(Y)|{\cal F}_{t+}^Y\vee {\cal F}_{t-}^X]=E[h(Y
)|{\cal F}_{t+}^Y]\label{rccmp}\end{equation}
for all $t>0$.
\end{lemma}

\begin{proof}
Since ${\cal F}_{(t,\epsilon )}^X={\cal F}_{t-}^X$, RC-compatibility implies
\[E[h(Y)|{\cal F}_{(t,\epsilon )}^Y\vee {\cal F}_{t-}^X]=E[h(Y)|{\cal F}_{
(t,\epsilon )}^Y].\]
Taking the limit $\epsilon\rightarrow 0$, we have 
\[E[h(Y)|\cap_{\epsilon >0}({\cal F}_{(t,\epsilon )}^Y\vee {\cal F}_{
t-}^X)]=E[h(Y)|{\cal F}_{t+}^Y].\]
Since $\cap_{\epsilon >0}({\cal F}_{(t,\epsilon )}^Y\vee {\cal F}_{
t-}^X)\supset {\cal F}_{t+}^Y\vee {\cal F}_{t-}^X\supset {\cal F}_{
t+}^Y$, conditioning both 
sides on ${\cal F}_{t+}^Y\vee {\cal F}_{t-}^X$ gives (\ref{rccmp}).

Now assuming (\ref{rccmp}) holds for all $t>0$, we have
\[E[h(Y)|{\cal F}_{(t+s)+}^Y\vee {\cal F}_{t+s-}^X]=E[h(Y)|{\cal F}_{
t+s+}^Y],\]
and letting $s\rightarrow\epsilon -$, we have 
\begin{equation}E[h(Y)|\vee_{s<\epsilon}({\cal F}_{(t+s)+}^Y\vee 
{\cal F}_{t+s-}^X)]=E[h(Y)|{\cal F}_{t+\epsilon -}^Y]=E[h(Y)|{\cal F}^
Y_{(t,\epsilon )}].\label{rc3}\end{equation}
Since 
\[\vee_{s<\epsilon}{\cal F}_{(t+s)+}^Y\vee {\cal F}_{t+s-}^X\supset 
{\cal F}^Y_{(t+\epsilon )-}\vee {\cal F}_{t+\epsilon -}\supset {\cal F}^
Y_{(t,\epsilon )}\vee {\cal F}_{(t,\epsilon )}^X,\]
conditioning both sides of (\ref{rc3}) on ${\cal F}^Y_{(t,\epsilon 
)}\vee {\cal F}_{(t,\epsilon )}^X$ gives 
the desired result.
\end{proof}

\begin{example}{\rm
An Euler approximation gives a natural approach to 
proving existence of compatible or RC-compatible solutions for
\begin{equation}X(t)=U(t)+\int_0^tH(X,s-)dV(s).\label{smgsde2}\end{equation}
Set $\eta_n(t)=\frac {[nt]}n$, and et $U_n=U\circ\eta_n$ and $V_n
=V\circ\eta_n$.  Then 
existence of a solution $X_n$ of
\begin{equation}X_n(t)=U_n(t)+\int_0^tH(X_n,s-)dV_n(s),\label{smgsde2n}\end{equation}
is immediate and $X_n$ is adapted to $\{{\cal F}_t^Y\}$.  It follows that 
$X_n$ is both temporally compatible and 
RC-compatible with $Y$.  Theorem 5.4 of \cite{KP91} 
gives conditions on $H$ that ensure the convergence of 
$(U_n,V_n,X_n)$ to $(U,V,X)$ satisfying (\ref{smgsde2}).  
Lemma \ref{cmpexist} then ensures that $X$ is temporally 
compatible with $Y=(U,V)$, if $Y$ has no fixed points of 
discontinuity, or at least
RC-compatible with $Y$.
}
\end{example}

\begin{example}\label{bw}{\rm
Let $T>0$ and $Y=(U,V)$ be a process in $D_{{\Bbb R}^m\times {\Bbb R}^
d}[0,T]$.  
Let $f$ be a measurable function
\[f:[0,T]\times D_{{\Bbb R}^m}[0,T]\times D_{{\Bbb R}^d}[0,T]\rightarrow 
{\Bbb R}^m\]
satisfying $f(t,x,v)=f(t,x(\cdot\vee t),v)$ for each 
$(t,x,v)\in [0,T]\times D_{{\Bbb R}^m}[0,T]\times D_{{\Bbb R}^d}[
0,T]$.  
Following \citet*{BER05}, 
we consider the backward stochastic differential equation
\[X(t)=U(t)+E[\int_t^Tf(s,X,V)ds|{\cal F}_t^Y\vee {\cal F}_t^X],\]
where \cite{BER05} requires Condition \ref{mgcnd}.  We 
will require $X$ to be temporally compatible with 
$Y$, or if $Y$ has fixed points of  discontinuity, that $X$ be 
RC-compatible with $Y$. Setting $X_n(t)=U(T)$ for $t\geq T$, 
there exist solutions to the approximating problems
\[X_n(t)=U(t)+E[\int_t^Tf(s,X_n(\cdot +\frac 1n),V)ds|{\cal F}_t^
Y].\]
Assume that $|f(s,x,v)|\leq g(s,v)$ and $E[\int_0^Tg(s,V)ds]<\infty$. Set
\[Z_n(t)=E[\int_t^Tf(s,X_n(\cdot +\frac 1n),V)ds|{\cal F}_t^Y].\]
Recalling the definition of conditional variation, we have
\[V_T(Z_n)\equiv\sup_{\{t_i\}}E[\sum_i|E[Z_n(t_{i+1})-Z_n(t_i)|{\cal F}_{
t_i}^Y]|]\leq E[\int_0^Tg(s,V)ds],\]
where the $\sup$ is over all partitions of $[0,T]$.
We also have 
\[\sup_{0\leq t\leq T}|Z_n(t)|\leq\sup_{0\leq t\leq T}E[\int_0^Tg
(s,V)ds|{\cal F}_t^Y]<\infty\quad a.s.,\]
so the sequence $\{Z_n\}$ satisfies the Meyer-Zheng 
conditions (see \cite{MZ84,Kur91}), or  more precisely, 
$\{Z_n\}$ is relatively compact in the Jakubowski topology 
(see \cite{Jak97}).  The Jakubowski topology is not 
metrizable, but versions of the Prohorov theorem and the Skorohod 
representation theorem still hold.  See Theorem 1.1 of 
\cite{Jak97}.  We will denote the 
space of  cadlag functions under the Jakubowski topology 
by $D_E^{{\cal J}}[0,T]$.

Convergence in the Jakubowski 
topology implies convergence in measure, that is
convergence in the metric $d_m(x,y)=\int_0^T|x(s)-y(s)|\wedge 1ds$ 
which is used in the original paper, 
\cite{MZ84}, and in \cite{BER05}.  Relative compactness 
of $\{Z_n\}$ in $D_{{\Bbb R}^m}^{{\cal J}}[0,T]$ implies relative compactness of 
$(Z_n,Y)$ in $D^{{\cal J}}_{{\Bbb R}^m\times {\Bbb R}^m\times {\Bbb R}^
d}[0,T]$.  In contrast to the 
Skorohod topology (that is, the Skorohod $J_1$ topology),
\[D^{{\cal J}}_{{\Bbb R}^m\times {\Bbb R}^m\times {\Bbb R}^d}[0,T
]=D^{{\cal J}}_{{\Bbb R}^m}[0,T]\times D^{{\cal J}}_{{\Bbb R}^m\times 
{\Bbb R}^d}[0,T].\]

Addition is continuous 
in the Jakubowski topology, so if $(Z_n,Y)$  converges, 
then setting
$X_n=U+Z_n$, $(X_n,Z_n,Y)$ converges.  If $X_n$ converges to $X$, then 
$X_n(\cdot +\frac 1n)$ converges to $X$ and for all but at most 
countably many $t$, $X_n(t)$ converges to $X(t)$. 

For each $t\in [0,T]$, assume that the mapping 
\[(x,v)\in D^{{\cal J}}_{{\Bbb R}^m\times {\Bbb R}^d}[0,T]\rightarrow
\int_t^Tf(s,x,v)ds\in {\Bbb R}\]
is continuous.  Assume that we have selected a 
subsequence such that $(X_n,Y)\Rightarrow (X,Y)$.  By Theorem 
3.11 of \cite{Jak97} there exists a countable set $D$ such 
that for $\{t_i\}\subset [0,T]\setminus D$
\[(X_n(t_1),\ldots ,X_n(t_k),Y(t_1),\ldots ,Y(t_k),X_n,Y)\Rightarrow 
(X(t_1),\ldots ,X(t_k),Y(t_1),\ldots ,Y(t_k),X_n,Y)\]
in $({\Bbb R}^m)^k\times ({\Bbb R}^{m+d})^k\times D^{{\cal J}}_{{\Bbb R}^
m\times {\Bbb R}^{m+d}}[0,T]$.

Let $g_i\in C_b({\Bbb R}^{2m+d})$.  Then for $0\leq t_1<\cdots <t_
k\leq t$, 
$\{t_i\},t\in [0,T]\setminus D$,
\begin{eqnarray*}
0&=&E[(X_n(t)-U(t)-\int_t^Tf(s,X_n,V)ds)\prod_{i=1}^kg_i(X_n(t_i)
,Y(t_i))]\\
&\rightarrow&E[(X(t)-U(t)-\int_t^Tf(s,X,V)ds)\prod_{i=1}^kg_i(X(t_
i),Y(t_i))].\end{eqnarray*}
Note that since
\[|X_n(t)-U(t)|\leq E[\int_0^Tg(s,V)ds|{\cal F}_t^Y],\]
$\{X_n(t)-U(t)\}$ is uniformly integrable justifying the 
convergence of the expectations.2
It follows that for each $t\in [0,T]\setminus D$,
\[X(t)=U(t)+E[\int_t^Tf(s,X,V)ds)|{\cal F}_t^X\vee {\cal F}_t^Y],\]
and the identity extends to all $t\in [0,T]$ by the right 
continuity of  $X$ and $U$.

If $Y$ has no fixed points of discontinuity, then $X$ has no 
fixed points of discontinuity and $X$ is  temporally 
compatible with $Y$.  In any case, $X$ is RC-compatible with 
$Y$.
}

\end{example}

\begin{example}{\rm 
The multiple time-change equation 
\begin{equation}X(t)=X(0)+\sum_{k=1}^mW_k(\int_0^t\beta_k(X(s))ds
)\zeta_k+\int_0^tF(X(s))ds,\label{mtceq}\end{equation}
arises naturally in the derivation of diffusion 
approximations for continuous time Markov chains.  (See, 
for example, \cite{EK86}, Chapter11.)  Here the $W_k$ are 
independent, scalar, standard Brownian motions, $X(0)$ is 
a ${\Bbb R}^d$-valued random variable independent of the $W_k$, $
\zeta_k\in {\Bbb R}^d$, 
and the $\beta_k$ and $F$ are measurable functions (typically 
continuous) satisfying $\beta_k:{\Bbb R}^d\rightarrow [0,\infty )$ and 
$F:{\Bbb R}^d\rightarrow {\Bbb R}^d$.  
Setting $Y=(X(0),W_1,\ldots ,W_m)$ and 
$\tau_k(t)=\int_0^t\beta_k(X(s))ds$, for $\alpha\in [0,\infty )^m$,
 define 
\[{\cal F}_{\alpha}^Y=\sigma (W_k(s_k):0\leq s_k\leq\alpha_k,k=1,
\ldots ,m)\vee\sigma (X(0))\]
 and
\[{\cal F}_{\alpha}^X=\sigma (\{\tau_1(t)\leq s_1,\tau_2(t)\leq s_
2,\ldots \}:s_i\leq\alpha_i,i=1,2,\ldots ,t\geq 0).\]

If the $\beta_k$ are continuous, $\{{\cal F}_{\alpha}^Y\}$ and $\{
{\cal F}_{\alpha}^X\}$ determine a 
compatibility condition satisfying the conditions of 
Lemma \ref{cmpexist}.  

If $X$ is a compatible solution, then $\tau (t)=(\tau_1(t),\ldots 
,\tau_m(t))$ 
is a stopping time with respect to $\{{\cal F}_{\alpha}^X\vee {\cal F}_{
\alpha}^Y\}$ and 
$W_k(\int_0^t\beta_k(X(s))ds)$, $k=1,\ldots ,m$, are $\{{\cal F}_{
\tau (t)}\}$-martingales.  It 
follows that  $X$ is a solution of the martingale problem for 
\[Af(x)=\frac 12\sum_{i,j}a_{ij}(x)\partial_i\partial_jf(x)+F(x)\cdot
\nabla f(x),\]
$a(x)=\sum_{k=1}^m\beta_k(x)\zeta_k\zeta_k^T$.  (Note that $m$ may be infinity 
provided $\sum_{k=1}^{\infty}\beta_k(x)|\zeta_k|^2<\infty$.)

Setting $\eta_n(t)=\frac {[nt]}n$, 
\[X_n(t)=X(0)+\sum_{k=1}^mW_k(\int_0^{\eta_n(t)}\beta_k(X_n(s))ds
)\zeta_k+\int_0^{\eta_n(t)}F(X_n(s))ds\]
has a unique piecewise constant solution that has the 
same distribution as the usual Euler approximation to 
the corresponding It\^o equation. Under appropriate growth 
conditions on the $\beta_k$ and $F$ (for example, if the $\beta_k$ and $
F$ 
are bounded), $\{X_n\}$ is relatively compact for 
convergence in distribution in $D_{{\Bbb R}^d}[0,\infty )$, 
and if the $\beta_k$ 
and $F$ are continuous, any limit point $X$ of $\{X_n\}$ will 
satisfy (\ref{mtceq}). Lemma \ref{cmpexist} gives that $X$ is 
compatible with $Y$. 

Uniqueness of the distribution of $X$ would follow from 
uniqueness for the corresponding martingale problem; 
however, except for $m=1$, no pathwise uniqueness result 
of any generality is known.  
Let $\tau_k(t)=\int_0^t\beta_k(X(s))ds$ and $\gamma (t)=\int_0^tF
(X(s))ds$.  Then
\begin{eqnarray*}
\dot{\tau}_l(t)&=&\beta_l(X(0)+\sum_kW_k(\tau_k(t))\zeta_k+\gamma 
(t))\\
\dot{\gamma }(t)&=&F(X(0)+\sum_kW_k(\tau_k(t))\zeta_k+\gamma (t))
,\end{eqnarray*}
which is a random ordinary differential equation.  Except 
in the case $\beta_k$ all constant, however, the right side is at best 
H\"older of order 1/2.

}
\end{example}

\bibliography{wkstr}

\end{document}